\documentclass[12pt]{amsart}

\usepackage[english]{babel}
\usepackage{times,amsfonts,amsmath,amssymb,dsfont,mathrsfs,amsthm} 
\usepackage{pdfsync}
\usepackage{graphicx} 
\usepackage{mathabx}
\usepackage{cases}

\usepackage{color}

\definecolor{gr}{rgb}   {0.,   0.69,   0.23 }
\definecolor{bl}{rgb}   {0.,   0.5,   1. }
\definecolor{mg}{rgb}   {0.85,  0.,    0.85}
\definecolor{or}{rgb}   {0.9,  0.5,   0.}

\definecolor{webred}{rgb}{0.75,0,0}
\definecolor{webgreen}{rgb}{0,0.75,0}
\usepackage[citecolor=webgreen,colorlinks=true,linkcolor=webred]{hyperref}

\newtheorem{theorem}{Theorem}[section]
\newtheorem{proposition}[theorem]{Proposition}
\newtheorem{lemma}[theorem]{Lemma}
\newtheorem{corollary}[theorem]{Corollary}

\theoremstyle{definition}
\newtheorem{definition}[theorem]{Definition}

\theoremstyle{remark}
\newtheorem{remark}[theorem]{Remark}

\newcommand{\beq}{\begin{equation}}
\newcommand{\eeq}{\end{equation}}

\newcommand{\Bk}{\color{black}}

\newcommand{\N}{\mathbb{N}}
\newcommand{\R}{\mathbb{R}}

\newcommand{\rd}{{\mathrm d}}
\newcommand{\one}{\mathds{1}}

\newcommand{\dS}{{\mathbb{S}}}

\newcommand{\cB}{\mathcal{B}}
\newcommand{\cC}{\mathcal{C}}

\newcommand{\cO}{\mathcal{O}}
\newcommand{\cP}{\mathcal{P}}
\newcommand{\cQ}{\mathcal{Q}}

\newcommand{\cU}{\mathcal{U}}
\newcommand{\cV}{\mathcal{V}}

\newcommand{\gC}{\mathfrak{C}}
\newcommand{\gD}{\mathfrak{D}}

\newcommand{\gP}{\mathfrak{P}}
\newcommand{\gS}{\mathfrak{S}}

\newcommand{\bel}{\begin{equation} \label}
\newcommand{\ee}{\end{equation}}

\newcommand{\dD}{\mathbb{D}}

\newcommand{\dx}{\mathbb{X}}

\newcommand{\sE}{\mathscr{E}}
\newcommand{\seE}{\mathscr{E}^{*}}

\newcommand{\rG}{\mathrm{G}}
\newcommand{\rJ}{\mathrm{J}}

\newcommand{\rU}{\mathrm{U}}

\newcommand\eps{\epsilon}
\newcommand\spec{\gS}

\newcommand\supp{\operatorname{supp}}

\newcommand\diam{\operatorname{diam}}
\newcommand\Id{\operatorname{\mathbb{I}}}

\makeatletter
\def\section{\@startsection{section}{1}\z@{.9\linespacing\@plus\linespacing}%
  {.7\linespacing} {\fontsize{13}{15}\selectfont\scshape\centering}}
\def\paragraph{\@startsection{paragraph}{4}%
  \z@{0.3em}{-.5em}%
  {$\bullet$ \ \normalfont\itshape}}
\makeatother

\title[]{On the negative spectrum of the Robin Laplacian in corner domains}
\author{Vincent Bruneau}
\address{Universit\'e  de Bordeaux, IMB, UMR CNRS 5251, 351 cours de la lib\'eration, 33405 Talence Cedex, {\it Tel.} +33 (0) 5 40 00 21 32} \email{vbruneau@math.u-bordeaux1.fr}
\author{Nicolas Popoff}
\address{Institut math\'ematique de Bordeaux, Universit\'e Bordeaux 1,
351 cours de la lib\'eration, 33405 Talence Cedex, France }

\email{nicolas.popoff@math.u-bordeaux1.fr}

\date{\today}

\begin{document}
\maketitle

\begin{abstract}
 For a bounded corner domain $\Omega$, we consider the attractive Robin Laplacian in $\Omega$ with large Robin parameter.  Exploiting multiscale analysis and a recursive procedure, we have a precise description of the mechanism giving the bottom of the spectrum.  It allows also the study of the bottom of the essential spectrum on the associated tangent structures given by cones. Then we obtain the asymptotic behavior of the principal eigenvalue for this singular limit in any dimension, with remainder estimates.  The same method works for the Schr\"odinger operator in $\R^n$ with a strong attractive $\delta$-interaction supported on $\partial \Omega$. Applications to some Ehrling's type estimates and the analysis of the critical temperature of some superconductors are also provided. 
\end{abstract}

\tableofcontents

\section{Introduction}
\subsection{Context: Robin Laplacian with large parameter}\label{S:intro1}
Let $M$ be a Riemannian manifold without boundary and $\Omega$ an open domain of $M$ (in practice one may thinks to $M=\R^{n}$ or $M=\dS^{n}$). We assume that $\Omega$ is a bounded corner domain which belongs to the recursive class described in \cite{Dau88}.
We are interested in the following eigenvalue problem: 
\beq
\label{E:evproblem}
\left \{ 
\begin{aligned}
&-\Delta u =\lambda u \ \ \mbox{on}\ \  \Omega
\\
&\partial_{\nu}u-\alpha u=0 \ \ \mbox{on}\ \  \partial\Omega
\end{aligned}
\right. .
\ee
Here $\alpha\in \R$ is the Robin parameter and $\partial_{\nu}$ denotes the outward normal to the boundary of $\Omega$.

We denote by $\cQ_{\alpha}[\Omega]$ the quadratic form of the Robin Laplacian on $\Omega$ with parameter $\alpha$:
\beq
\label{E:DefcqQ}
\cQ_{\alpha}[\Omega](u):=\| \nabla u\|_{L^2(\Omega)}^2-\alpha \|u\|_{L^2(\partial\Omega)}^2, \quad u\in H^1(\Omega).
\ee
Since $\Omega$ is bounded and is the union of Lipschitz domains (see \cite[Lemma AA.9]{Dau88}), the trace injection from $H^{1}(\Omega)$ into $L^{2}(\partial\Omega)$ is compact and the quadratic form $\cQ_{\alpha}[\Omega]$ is lower semi-bounded. We define $L_{\alpha}[\Omega]$ its self-adjoint extension whose spectrum is a sequence of eigenvalues, and we denote by $\lambda(\Omega,\alpha)$ the first one. It is the principal eigenvalue of the system \eqref{E:evproblem}.
 
The study of the spectrum of $L_{\alpha}[\Omega]$ has received some attention in the past years, in particular for the singular limit 
$\alpha\to+\infty$.
This problem appeared first in a model of reaction diffusion for which the absorption mechanism competes with a boundary term \cite{LaOckSa98}, and more recently it is established  that the understanding of $\lambda(\Omega,\alpha)$ provides informations on the critical temperature of surface superconductivity under zero magnetic field \cite{GiSm07}. Let us also mention that such models are also used in quantum Hall effect and topological insulators to justify appearance of edges states (see \cite{AsBaPP15}).

 In view of the quadratic form, it is not difficult to see that $\lambda(\Omega,\alpha)\to-\infty$ as $\alpha\to+\infty$  (while in the limit $\alpha\to-\infty$ they converge to those of the Dirichlet Laplacian). When $\Omega\subset \R^{n}$ is smooth, there holds $\lambda(\Omega,\alpha)\leq -\alpha^2$ for all $\alpha\geq0$, see \cite[Theorem 2.1]{GiSm07}.
  More precisely, it is known that $\lambda(\Omega,\alpha) \sim C_\Omega \alpha^2$ as $\alpha\to+\infty$ for some particular domains: for smooth domains,  $C_\Omega=-1$ (see \cite{LaOckSa98,LouZhu04} and \cite{DaKe10} for higher eigenvalues), and for planar polygonal domains with corners of opening $(\theta_{k})_{k=1,\ldots,N}$, there holds $$C_\Omega= - \max_{ 0<\theta_{k}<\pi}(1,\sin^{-2} \tfrac{\theta_{k}}{2}).$$
  This last formula conjectured in \cite{LaOckSa98} is proved in \cite{LevPar08}. For general domains $\Omega$ having a piecewise smooth boundary it is natural to study the operator on tangent spaces and  from homogeneity reasons (see Lemma \ref{L:scalings}), one expects that $\lambda(\Omega,\alpha) \sim C_\Omega \alpha^2$, when $\alpha \rightarrow + \infty$, with some negative constant $C_\Omega$. In \cite{LevPar08}, Levitin-Parnovski  consider domains with corners satisfying the uniform interior cone condition. For each $x \in \partial \Omega$, they introduce $E(\Pi_x)$ the bottom of the spectrum of the Robin Laplacian on an infinite model cone $\Pi_x$ (if $x $ is a regular point, it is a half-space) and show 
 \begin{equation}
\label{E:RoughLimit}
\lim_{\alpha\to+\infty} \frac{\lambda(\Omega,\alpha)}{\alpha^2}=\inf_{x \in \partial \Omega} E(\Pi_x).
\end{equation}
But we have no guarantee concerning the finiteness of  $E(\Pi_x)$ and moreover, even if it is  finite, we don't know if their infimum, over $\partial \Omega$, is reached.
  Then an important question is to understand more precisely the influence of the geometry (of the boundary) of $\Omega$ in the asymptotic behavior of $\lambda(\Omega,\alpha)$ in order to give a meaning to \eqref{E:RoughLimit} (in particular proving that $\inf_{x \in \partial \Omega} E(\Pi_x)$ is finite)  and if possible, to obtain some remainder estimates.

\subsection{Local energies on admissible corner domains}
 In this article, our purpose is to develop a framework in the study of such asymptotics by introducing the {\it local energy}  function $x \mapsto E(\Pi_x)$   on the recursive class of corner domains (see \cite{Dau88}). The natural tangent structures are given by dilation invariant domains, more shortly referred as {\it cones}.
When the domain is a convenient cone $\Pi$, the quadratic form in \eqref{E:DefcqQ} may still be defined on $H^{1}(\Pi)$. By immediate scaling, $\cQ_{\alpha}[\Pi]$ is unitarily equivalent to $\alpha^2\cQ_{1}[\Pi]$. Therefore the case where the parameter is equal to 1 plays an important role and we denote by $\cQ[\Pi]=\cQ_{1}[\Pi]$. For a general cone, 
 we don't know whether $\cQ[\Pi]$ is lower semi-bounded, and we define 
$$E(\Pi)=\inf_{\substack{u\in H^{1}(\Pi) \\ u\neq0}} \frac{Q[\Pi](u)}{\|u\|^2}$$
the ground state energy of the Robin Laplacian on $\Pi$.  
For  $x\in  {\overline \Omega}$,  denote by $\Pi_{x}$ the tangent cone at $x$. When $\Pi_x$ is the full space (corresponding to interior points), there is no boundary and $E(\Pi_x)=0$, whereas,   when $\Pi_x$ is a half-space (corresponding to regular points of the boundary), it is easy to see that $E(\Pi_x)=E(\R_+)=-1$ (see \cite{DaKe10}). Moreover, when $\Pi_x$ is an infinite planar sector of opening $\theta$, denoted by $S_{\theta}$,  $E(\Pi_x)$ is given by (see \cite{LaOckSa98,LevPar08}): 
\beq
\label{E:Esectors}
 E(S_{\theta})=\left\{ \begin{aligned} &-\sin^{-2}\tfrac{\theta}{2} \ \ \mbox{if} \ \ \theta\in (0,\pi) \\ &-1 \ \ \mbox{if} \ \ \theta\in (\pi,2\pi) \end{aligned}\right. 
 \ee
No such explicit expressions are available for general cones in higher dimension. In link with \eqref{E:RoughLimit}, we introduce the infimum of local energy $E(\Pi_{x})$, for $x\in  {\overline \Omega}$, 
which, from the above remarks, is also the infimum on the boundary: 
\begin{equation}\label{defsE}
\sE(\Omega):=\inf_{x\in \partial {\Omega}}E(\Pi_{x}).
\end{equation}
 Our goal is to prove the finiteness  of $\sE(\Omega)$ (and firstly, of $E(\Pi_x)$, $x\in \overline {\Omega}$) for admissible corner domains and to give an estimate of $\lambda(\Omega,\alpha)-\alpha^2\sE(\Omega)$ as $\alpha$ is large.
 In view of the above particular cases, the local energy is clearly discontinuous (even for smooth domains it is piecewise constant with values in $\{0, -1\}$). 
 We will use a recursive procedure in order to prove the finiteness and the lower semi-continuity of the local energy in the general case. It relies also on a multiscale analysis to get an estimate of the first eigenvalue, as developed in the recent work \cite{BoDauPof14} for the semiclassical magnetic Laplacian. 
 Unlike the work  \cite{BoDauPof14}, where the complexity of model problems limits the study to dimension 3, for  the Robin Laplacian we have a good understanding of the ground state energy on corner domains in any dimension. Moreover these technics allow an analog spectral study  of Schr\"odinger operator with $\delta$-interaction supported on closed corner hypersurfaces and on conical surfaces. 

\subsection{Results for the Robin Laplacian}
 We define below generic notions associated with cones: 
\begin{definition}
\label{D:cones}
A cone $\Pi$ is a domain of $ \R^{n}$ which is dilation invariant: 
$$\forall x\in \Pi, \forall \rho>0, \quad \rho x\in \Pi.$$
The section of a cone $\Pi$ is $\Pi\cap \dS^{n-1}$, generically denoted by $\omega$. We said that two cones $\Pi_{1}$ and $\Pi_{2}$ are equivalent, and we denote by $\Pi_{1}\equiv \Pi_{2}$, if they can be deduced one from another by a rotation. Given a cone $\Pi$, there exists $0 \leq d \leq n$ such that
$$\Pi \equiv \R^{n-d}\times \Gamma, \ \ \mbox{with} \ \ \Gamma \ \mbox{a cone in } \R^{d}.$$
When $d$ is minimal for such an equivalence, we say that $\Gamma$ is the reduced cone of $\Pi$. When $d=n$, so that $\Pi=\Gamma$, we say that $\Pi$ is irreducible.
\end{definition}
In the following, $\gP_{n}$ denotes the class of admissible cones of $\R^{n}$, and $\gD(M)$ denotes the class of admissible corner domains on a given Riemannian manifold $M$ without boundary. We refer to Section \ref{S:CD} for precise definitions of these classes of domains. Note that these domains include various possible geometries, like polyhedra and circular cones.
\begin{theorem}
\label{T:Efiniessspectrum}
Let $\Pi\in \gP_{n}$  be an admissible cone. Then we have:
\begin{enumerate}
\item $E(\Pi)>-\infty$ and the Robin Laplacian $L[\Pi]$ is well defined as the Friedrichs extension of $Q[\Pi]$ with form domain $D(Q[\Pi])=H^1(\Pi)$. 
\item Let $\Gamma$ be the reduced cone of $\Pi$. Then the bottom of the essential spectrum of $L[\Gamma]$ is $\sE(\omega)$, where $\omega$ is the section of $\Gamma$.
\end{enumerate}
\end{theorem}



This theorem generalizes to cones having no regular section, the result of  \cite{Pank16} where the bottom of the essential spectrum is proved to be -1 for cone with regular section (as discussed at the end of Section \ref{S:intro1}, in this case  
$\sE(\omega)= -1$).

The crucial point of this theorem is to show that the Robin Laplacian on a cone, far from the origin, can be linked to the Robin Laplacian on the section of the cone, with a parameter related to the distance to the origin. 

Notice that this theorem provides an effective procedure in order to compute the bottom of the essential spectrum for Laplacians on cones.
In particular, as shown by Remark \ref{ContreExample}, we obtain that  \cite[Theorem 3.5]{LevPar08} is incorrect in dimension $n\geq 3$, indeed we construct a cone which contains an hyperplane passing through the origin for which  the bottom of the essential spectrum (then of the spectrum) of the Robin Laplacian is below -1.

The next step is to minimize the local energy on a corner domain $\Omega$ and to prove that $\sE(\Omega)$ is finite. 
Thanks to Theorem \ref{T:Efiniessspectrum}, we will be able to prove some monotonicity properties (on singular chains, see Section \ref{SS:Sc} for the definition) which, combined with continuity of the local energy (for the topology of singular chains), allow to apply \cite[Section 3]{BoDauPof14} and to obtain:
\begin{theorem}
\label{T:inffinite}
For any corner domain $\Omega\in \gD(M)$, the energy function $x\mapsto E(\Pi_{x})$ is lower semi-continuous, and we have $\sE(\Omega)>-\infty$.
\end{theorem}
To get an asymptotics of $\lambda(\Omega,\alpha)$ with control of the remainders, we need to control error terms when using change of variables and cut-off functions. However, the principal curvatures of the regular part of a corner domain may be unbounded in dimension $n\geq3$ (think of a circular cone), so the standard estimates when using approximation of metrics may blow up. We use a multiscale analysis to overcome this difficulty and we get the following result:
\begin{theorem}
\label{T:AsReste}
Let $\Omega\in \gD(M)$ with $n\geq2$ the dimension of $M$. Then there exists $\alpha_{0}\in \R$, two constants $C^{\pm}>0$ and two integers $0 \leq \overline{\nu} \leq \overline{\nu}_{+} \leq n-2$ such that
$$\forall \alpha \geq \alpha_{0}, \quad -C^{-}\alpha^{2 - \frac{2}{2\overline{\nu}_{+}+3}} \leq \lambda(\Omega,\alpha)-\alpha^2 \sE(\Omega) \leq C^{+} \alpha^{2-\frac{2}{2\overline{\nu}+3}}$$
\end{theorem}
The constant $\overline{\nu}$ corresponds to the degree of degeneracy of the curvatures near the minimizers of the local energy, its precise definition can be found in \eqref{D:nubarre}. The constant $\overline{\nu}_{+}$ describes the degeneracy of the curvatures globally in $\overline{\Omega}$, see Lemma \ref{L:partitionunity}. In particular, when $\Omega$ is polyhedral (that is a domain with bounded curvatures on the regular part), then $\overline{\nu}=\overline{\nu}_{+}=0$.

The proof of the lower bound relies on a multiscale partition of the unity where the size of the balls optimizes the error terms. 
The upper bound is less classical: using the concept of {\it singular chain},  we isolate a tangent "sub-reduced-cone"  for which the bottom of the spectrum corresponds to an isolated eigenvalue (below the essential spectrum). Then we construct recursive quasi-modes, coming from this  tangent "sub-reduced-cone". 

Note finally that for regular domains, more precise asymptotics involving the mean curvature can be found (\cite{Pank13,HeKa15} in dimension 2 and \cite{PankPof15,PankPof16} for higher dimension). A precise analysis is also done for particular polygonal geometries: tunneling effect in some symmetry cases (\cite{HePank15}), and reduction to the boundary when the domain is the exterior of convex polygon (\cite{Pank15}). In all these cases, the local energy is piecewise constant, and new geometric criteria appear near the set of minimizer. In fact, the local energy can be seen as a potential in the standard theory of the harmonic approximation (\cite{DiSj99}) and under additional hypothesis on the local energy, it is reasonable to expect more precise asymptotics in higher dimensions.
 For polygons (dimension 2), another approach would consist in comparing the limit problem to a problem on a graph, in the spirit of \cite{Gri08}, the nodes (resp. the edges) corresponding to the vertices (resp. the sides) of the polygons. But it is not clear how such an approach could be generalized in any dimension.

\subsection{Applications of the method for the Schr\"odinger operator with $\delta$-interaction}

 Let $\Omega \in \gD(M)$ be a corner domain and let $S=\partial \Omega$ be its boundary. We consider $L^{\delta}_{\alpha}[M,S]$ the self-adjoint extension associated with the quadratic form
$$\cQ^{\delta}_{\alpha}[M,S](u):=\| \nabla u\|_{L^2(M)}^2-\alpha \|u\|_{L^2(S)}^2, \quad u\in H^1(M).$$
The associated boundary problem is the Laplacian with the derivative jump condition  across the closed hypersurface $S$: $[\partial_{\nu} u]_{\partial \Omega} = \alpha u$.
It is well known (see e.g. \cite{BrExKuSe94}) that since $S$ is bounded, $L^{\delta}_{\alpha}[\R^n,S]$ is a relatively compact perturbation of $L_0= -\Delta$ on $L^2(\R^n)$ and then
$$\sigma_{ess}\Big( L^{\delta}_{\alpha}[\R^n,S] \Big) = \sigma_{ess}( L_0 )= [0, + \infty).$$
Moreover $L^{\delta}_{\alpha}[\R^n,S]$ has a finite number of negative eigenvalues. If we denote by $\lambda^{\delta}(S, \alpha)$ the lowest one, by applying our technics developed for the Robin Laplacian, all the above results are still valid replacing $\lambda(\Omega, \alpha)$ by $\lambda^{\delta}(S, \alpha)$. In particular for $x\in S$, the tangent cone to $\Omega$ at $x$ is $\Pi_{x}$, and its boundary is denoted by $S_{x}$. We still define the tangent operator as $L^{\delta}_{1}[\R^n,S_{x}]$, and the associated local energy at $x$, $E^{\delta}(S_{x})$, and their infimum $\sE^{\delta}(S)$. Then 
\begin{theorem}
Theorems \ref{T:Efiniessspectrum}--\ref{T:AsReste} remain valid when replacing the Robin Laplacian $L_{\alpha}[\Omega]$ by the $\delta$-interaction Laplacians $L^{\delta}_{\alpha}[M,S]$, $\lambda(\Omega, \alpha)$ by $\lambda^{\delta}(S, \alpha)$, $E(\Pi_{x})$ by $E^{\delta}(S_{x})$ and $\sE(\Omega)$ by $\sE^{\delta}(S)$.
\end{theorem}
When $x$ belongs to the regular part of $S$, $S_{x}$ is an hyperplane and 
\beq
\label{E:}
E^{\delta}(\R^n,S_{x})=E^{\delta}(\R, \{0\})=-\tfrac14,
\ee
see \cite{ExYo02}. Therefore $\sE^{\delta}(S)=-\frac14$ when $S$ is regular, and we obtain the known main term of the asymptotic expansion of $\lambda^{\delta}(S, \alpha)$ proved in dimension $2$ or $3$ (see \cite{ExYo02,  ExPa14, DiExKuPa15}). 

To our best knowledge the only studies for $\delta$-interactions supported on non smooth hypersurfaces are for broken lines and conical domains with circular section (see \cite{BerExLot14, DuRa14,ExKon15,LotOur16}). In that case, we clearly have $\sigma( L^{\delta}_{\alpha}[\R^n,S] ) =   \alpha^2 \; \sigma( L^{\delta}_{1}[\R^n,S] )$ (see Lemma \ref{L:scalings}), and it is proved in the above references that the bottom of the essential spectrum of $L^{\delta}[\R^n,S]$ is $- 1/4$. In view of our result, it remains true  when the section of the conical surface is smooth. 
Moreover, our work seems to be the first result giving the main asymptotic behavior of $ \lambda^{\delta}(S, \alpha)$ for interactions supported by general closed hypersurfaces with corners.

\begin{remark}
\label{R:weight}
For the Robin Laplacian and  the $\delta$-interaction Laplacian, we can add a smooth positive weight function $G$ in the boundary conditions. These conditions become, for the Robin condition $\partial_\nu u = \alpha G(x)  u$ and for the $\delta$-interaction case, $[\partial_{\nu} u ] = \alpha G(x)  u$. In our analysis, for $x \in \partial \Omega$ fixed, we have only to change $\alpha$ into $\alpha G(x)$ and clearly, the results are still true by replacing  $\sE(\Omega)$ and $\sE^\delta(S)$ by:
$$ \sE_G( \Omega):=\inf_{x\in \partial \Omega}G(x)^2 E(\Pi_{x}), \qquad \sE_G^\delta(S):=\inf_{x\in S}G(x)^2 E^\delta (S_{x}).$$
For the Robin Laplacian, such cases were already considered in \cite{LevPar08} and \cite{CoGa11}.
\end{remark}

\subsection{Organisation of the article}
In Section \ref{S:CD}, we recall the definitions of corner domains, in the spirit of \cite{Dau88,MazPla77}, and we give some properties proved in \cite{BoDauPof14}. Section \ref{S:CV} is devoted to the effects of perturbations on the quadratic form of the Robin Laplacian. It contains several technical Lemmas used in the following sections.

Section \ref{S:LB} contains the proof of the lower bound of Theorem \ref{T:AsReste}. 
It is based on a multiscale analysis in order to counter-balance the possible blow-up of curvatures in corner domains. 
In particular it involves the lower bound $\liminf_{\alpha\to+\infty}\lambda(\Omega,\alpha)/\alpha^2 \geq \sE(\Omega)$
 in any dimension, which is also used in Sections \ref{S:TO} and \ref{S:Sel}. Notice that in Section \ref{S:LB}, at this stage of the analysis, the quantity $\sE(\Omega)$ is still not known to be finite, its finiteness will be the recursive hypothesis of the next two sections.

Section \ref{S:TO} is a step in a recursive proof of  Theorem \ref{T:inffinite} developed in Section \ref{S:Sel}. Then when the finiteness of $\sE( \Omega)$ is  stated, Theorem \ref{T:Efiniessspectrum} is a direct consequence of Lemmas   \ref{L:coneavecHyppoH1} and \ref{L:specessprove} (see the end of Section \ref{SS:6.1}).

In Section \ref{S:UB}, we prove the upper bound of Theorem \ref{T:AsReste}. This is done by exploiting the results of Section \ref{S:Sel} in order to find a tangent problem that admits an eigenfunction associated with $\sE(\Omega)$. Then we construct recursive quasi-modes, qualified either of {\it sitting} or {\it sliding}, from the denomination of \cite{BoDauPof14}. 

 In Section \ref{S:Applications} we give two possible applications of our results.  A purely mathematical one concerns optimal estimates in compact injections of Sobolev spaces. In the second one we recall how from the study of $\lambda(\Omega, \alpha)$ we derive properties on the critical temperature for zero fields for systems with enhanced surface superconductivity (where  $\alpha^{-1}$ is related to the penetration depth).

\section{Corner domains}
\label{S:CD}
Here we give some background of so-called admissible corner domains, see \cite{Dau88,BoDauPof14}.
\subsection{Tangent cones and recursive class of corner domains}
Let $M$ be a Riemannian manifold without boundary. We define recursively the class of admissible corner domains $\gD(M)$ and admissible cones $\gP_{n}$, in the spirit of \cite{Dau88}: 

Initialization: $\gP_0$ has one element, $\{0\}$.
$\gD(\dS^0)$ is formed by all non empty subsets of $\dS^0$.

Recurrence: For $n\geq1$,
\begin{enumerate}
\item A cone $\Pi$ (see definition \ref{D:cones}) belongs to \Bk $\gP_n$ if and only if the section of $\Pi$ belongs to $\gD(\dS^{n-1})$,
\item $\Omega\in\gD(M)$ if and only if $\Omega$ is bounded, and for any $x\in\overline\Omega$, 
there exists a tangent cone $\Pi_{x}\in\gP_n$ to $\Omega$ at $x$.
\end{enumerate}
By definition, $\Pi_{x}$ is the tangent cone to $\Omega$ at $x\in \overline{\Omega}$ if there exists a local map $\psi_{x}:\cU_{x}\mapsto \cV_{x}$ where $\cU_{x}$ and $\cV_{x}$ are neighborhoods (called {\it map-neighborhoods})
of $x$ in $M$ and of 0 in $\R^{n}$ respectively, and $\psi_{x}$ is a diffeomorphism such that
\begin{equation}
\label{eq:diffeo}
\psi(x)=0
\quad\mbox{and}\quad
(\rd \psi)(x)=\Id
\quad\mbox{and}\quad
   \psi_{x}(\cU_{x}\cap\Omega) = \cV_{x}\cap\Pi_{x} 
   \quad\mbox{and}\quad
    \psi_{x}(\cU_{x}\cap\partial\Omega) = \cV_{x}\cap\partial\Pi_{x} .
\end{equation}
In dimension 2, cones are half-planes, sectors and the full plane. The corner domains are (curvilinear) polygons on $M$ with a finite number of vertices, each one of opening in $(0,\pi)\cup(\pi,2\pi)$. This includes of course regular domains. 

The key quantity in order to estimate errors when making a change of variables is
\beq
\label{D:curvature}
\kappa(x)=\|\rd\psi\|_{W^{1,\infty}(\cU_{x})}.
\ee
 It depends not only on $x$, but also on the choice of the local map. Note that, unlike for a regular domain, the curvature of the regular part of a corner domain may be unbounded (think of a circular cone). Therefore, $\kappa(x)$ is not bounded in general when picking an atlas of $\overline{\Omega}$. An important subclass of corner domains are those who are {\it polyhedral}: A cone is said to be polyhedral if its boundary is contained in a finite union of hyperplanes, and a domain is said polyhedral if all its tangent cones are polyhedral.
 
  As proven in \cite{BoDauPof14}, for a polyhedral domain, it is possible to find an atlas such that $\kappa$ is bounded. In the general case, we will have to control the possible blow-up of $\kappa$.

A list of examples can be found in \cite[Section 3.1]{BoDauPof14}. Let us recall that in dimension 2, all cones are polyhedral and therefore so are all corner domains,
but this is not true anymore when $n\geq3$: circular cones are typical examples of cones which are not polyhedral.

\subsection{Singular chains}
\label{SS:Sc}
For $x_{0}\in\overline{\Omega}$, we denote by $\Gamma_{x_{0}}\in \gP_{d_{0}}$ the reduced cone of $\Pi_{x_{0}}$, see Definition \ref{D:cones}, and $\omega_{x_{0}}$ the section of $\Gamma_{x_{0}}$. A singular chain $\dx=(x_{0},\ldots,x_{p})$ is a sequence of points, with $x_{0}\in \overline{\Omega}$, $x_{1}\in \overline{\omega_{x_{0}}}$, and so on. We denote by $\gC(\Omega)$ the set of singular chains (in $\Omega$), $\gC_{x_{0}}(\Omega)$ the set of chains initiated at $x_{0}$ and $\gC_{x_{0}}^{*}(\Omega)$ the set of $\dx\in \gC_{x_{0}}(\Omega)$ such that $\dx\neq(x_{0})$. We denote by $l(\dx)$ the integer $p+1$ which is the length of the chain. Note that $1\leq l(\dx) \leq n+1$, and that $l(\dx) \geq2$ when $\dx\in \gC^{*}_{x_{0}}(\Omega)$. 

With a chain $\dx$ is canonically associated a cone, denoted by $\Pi_{\dx}$, called a tangent structure: 
\begin{itemize}
\item When $\dx=(x_{0})$, then $\Pi_{\dx}=\Pi_{x_{0}}$.
\item When $\dx=(x_{0},x_{1})$, write as above, in some adapted coordinates, $\Pi_{x_{0}}=\R^{n-d_{0}}\times \Gamma_{x_{0}}$. Let $C_{x_{1}}$ be the tangent cone to $\omega_{x_{0}}$ at $x_{1}$. Then, in the adapted coordinates, $\Pi_{\dx}=\R^{n-d_{0}}\times \langle x_{1} \rangle \times C_{x_{1}}$, where $\langle x_{1} \rangle$ is the vector space spanned by $x_{1}$ in $\Gamma_{x_{0}}$.
\item And so on for longer chains.
\end{itemize}

 We refer to \cite[Section 3.4]{BoDauPof14} for complete definitions. Since singular chains are one of the tools of our analysis, we provide below some examples for a better understanding. In these examples, we assume for simplicity that $\Pi_{x_{0}}$ is irreducible 
\begin{itemize}
\item If $x_{1}\in \Pi_{x_{0}}$ (interior point), then $\Pi_{(x_{0},x_{1})}$ is the full space.
\item If $x_{1}$ is in the regular part of the boundary of $\omega_{x_{0}}$, then $C_{x_{1}}$ is a half-space of $\R^{n-1}$, and $\Pi_{(x_{0},x_{1})}$ is a half-space of $\R^{n}$. In particular for a cone with regular section, all chains of length 2 are associated either with a half-space or the full space. The chains of length 3 are associated with the full space, and there are no longer chains.
\item If $\Pi_{x_{0}}\subset\R^{3}$ is such that its section is a polygon and if $x_{1}$ is one of its vertices, then $C_{x_{1}}$ is a two-dimensional sector, and $\Pi_{(0,x_{1})}$ is a wedge. If $x_{2}$ is on the boundary of the sector $C_{x_{1}}$, then $\Pi_{(x_{0},x_{1},x_{2})}$ is a half-space, but if $x_{2}$ is on the interior of the sector, then $\Pi_{(0,x_{1},x_{2})}=\R^{3}$.
\end{itemize}
Given a cone $\Pi\in \gP_{n}$, we will also consider chains of $\Pi$, for example chains in $\gC_{0}(\Pi)$ are of the form $(0,x_{1},\ldots)$, where $x_{1}$ belongs to the closure of the section of the reduce cone of $\Pi$.

The main idea is to consider the local energy as a function not only defined on $\overline{\Omega}$, but also on singular chains: $\gC(\Omega) \ni \dx\mapsto E(\Pi_{\dx})$. In order to show regularity properties of this function, we define a partial order on singular chain: we say that $\dx \leq \dx '$ if $l(\dx) \leq l(\dx')$, and for all $k \leq l(\dx)$, there hold $x_{k}=x_{k}'$. We also define a distance between cones through the action of isomorphisms:
\beq
\label{D:dd}
   \dD(\Pi,\Pi') =\frac12\bigg\{
   \min_{\substack {L\in\mathsf{BGL}_{n} \\ L\Pi = \Pi'}}
   \|L-\Id_n\|
      +\min_{\substack {L\in\mathsf{BGL}_{n} \\ L\Pi' = \Pi}}
   \|L-\Id_n\| \bigg\}\,,
\ee
where $\mathsf{BGL}_{n}$ is the ring of linear isomorphisms $L$ of $\R^{n}$ with norm $\|L\|\le1$. Note that by definition the distance between two cones is $+\infty$ if they do not belong to the same orbit for the action of $\mathsf{BGL}_{n}$ on $\gP_{n}$.

We define then the natural distance inherited on $\gC(\Omega)$ by $\dD(\dx,\dx')=\|x_{0}-x_{0}'\|+\dD(\Pi_{\dx},\Pi_{\dx'})$, see \cite[Definition 3.22]{BoDauPof14}.
Then \cite[Theorem 3.25]{BoDauPof14} states that any function $F:\gC(\Omega)\mapsto \R$, monotonous and continuous with respect to $\dD$, is lower semi-continuous when restricted to $\overline{\Omega}$ (which corresponds to chains of length 1). We will show these two criterion, see Corollary \ref{C:continuity_chains} and \ref{C:monotonicity}.

\section{Change of variables and perturbation of the metric}
\label{S:CV}
 This section contains mainly technical lemmas, which are useful in the following sections. \Bk  We define the operator with metric and we show the influence of a change of variable from a corner domains toward tangent cones on the quadratic form.

\subsection{Change of variables and operator with metrics}
We need to know how a change of variables transforms the quadratic form of the Robin Laplacian. Indeed we will consider diffeomorphisms $\psi : \cO \mapsto \cO'$, where $\cO$ and $\cO'$ are open set 
 in the following two situations: 
\begin{itemize}
\item Either $\cO$ and $\cO'$ will be cones in $\gP_{n}$ and $\psi$ will be a linear application of $\R^{n}$,
\item or $\cO$ and $\cO'$ will be map-neighborhoods, respectively of a point in a closure of a corner domain, and of 0 in the associated tangent cone.
\end{itemize}
This changes of variables will induce a regular metric $\rG: \cO'\mapsto GL_{n}$. In the case where $\psi$ is linear, $\rG$ will be constant.

 Let $L^{2}_{\rG}(\cO')$ be the space of the square-integrable functions for the weight $|\rG|^{-1/2}$, endowed with its natural norm $\|v\|_{L^2_{G}}:= \int_{\cO'}|v|^2 |\rG|^{-1/2}$. Due to the previous hypotheses, $L^{2}_{\rG}(\cO')=L^{2}(\cO')$. Let $g=\rG_{|\partial\cO'}$ be the restriction of the metric to the boundary. We introduce the quadratic form
 $$\cQ_{\alpha}[\cO',\rG](v)=\int_{\cO'}\langle \rG\nabla v,\nabla v\rangle |\rG|^{-1/2} -\alpha \int_{\partial\cO'}|v|^2|g|^{-1/2}.$$
Due to the above hypotheses on $\cO'$ and $\rG$, we can define this quadratic form on $H^{1}(\cO')$, endowed with the weighted norm $\|\cdot\|_{L^2_{G}}$.

\begin{lemma}
\label{L:CV}
Let $\cO$ and $\cO'$ be open sets and $\psi:\cO\mapsto \cO'$ a diffeomorphism as above. Let $\rJ:=\rd (\psi^{-1})$ be the Jacobian of $\psi^{-1}$ and $\rG:=\rJ^{-1}(\rJ^{-1})^{\top}$ the associated metric. Then for all $u\in H^{1}(\cO)$, there holds
$$\cQ_{\alpha}[\cO](u)=\cQ_{\alpha}[\cO',\rG](u\circ \psi^{-1}) \quad \mbox{and} \quad \|u\|_{L^2(\cO)}=\|u\circ \psi^{-1}\|_{L^2_{G}(\cO')}.$$ 
\end{lemma}
Said differently, if we define $\rU: u\mapsto u\circ\psi^{-1}$, then $\rU$ is an isometry from $L^2(\cO)$ onto $L^2_{G}(\cO')$, and $\cQ_{\alpha}[\cO',\rG]\rU=\cQ_{\alpha}[\cO]$. We will also use scaling on cones: 
\begin{lemma}
\label{L:scalings}
Let $\Pi$ be a cone, and $u\in H^{1}(\Pi)$. For $\alpha>0$, we define $u_{\alpha}(x):=\alpha^{-\frac{n}{2}}u(\frac{x}{\alpha})$. Then
$$\|u_{\alpha}\|_{L^2}= \|u\|_{L^2}\quad \mbox{and}\quad \cQ_{\alpha}[\Pi](u)=\alpha^2\cQ[\Pi](u_{\alpha}).$$
In particular, $\cQ_{\alpha}[\Pi]$ and $\alpha^2\cQ[\Pi]$ are unitarily equivalent.
\end{lemma}

\subsection{Approximation of metrics}
We will be led to consider situations where $\rJ-\Id$ is small (and so is $\rG-\Id$). Therefore, for $v\in H^{1}(\cO')$, we compute
$$\cQ_{\alpha}[\cO',\rG](v)-\cQ_{\alpha}[\cO'](v)=\int_{\cO'}\langle (\rG-\Id)\nabla v,\nabla v\rangle |\rG|^{-1/2}+ \int_{\cO'} |\nabla v|^2(|\rG|^{-1/2}-1)+\alpha\int_{\partial\cO'}|v|^2(|g|^{-1/2}-1) $$
and therefore 
\begin{multline*}
|\cQ_{\alpha}[\cO',\rG](v)-\cQ_{\alpha}[\cO'](v)| \leq  \left(\|\rG-\Id\|_{L_v^{\infty}}(\||\rG|^{-1/2}-1\|_{L_v^{\infty}}+1)+\||\rG|^{-1/2}-\Id\|_{L_v^{\infty}} \right) \|\nabla v \|_{L^2}^2
\\+\alpha \||g|^{-1/2}-1\|_{L_v^{\infty}}\|v\|_{L^2(\partial\cO')}
\end{multline*}
where $\|\cdot \|_{L_v^\infty}$ denotes the $L^{\infty}$ norm on $\supp(v)$. Assume now that $\|\rJ-\Id\|_{L_v^\infty}\leq 1$, then there exists a universal constant $C>0$ such that 
\begin{equation}
\label{E:UBRough}
|\cQ_{\alpha}[\cO',\rG](v)-\cQ_{\alpha}[\cO'](v)| \leq C \|\rJ-\Id\|_{L_v^\infty} (\|\nabla v \|_{L^2}^2+\alpha \|v\|_{L^2(\partial\cO')}).
\end{equation}
This may be written as 
\begin{align*}
(1- C \|\rJ-\Id\|_{L_v^\infty})\|\nabla v \|_{L^2}^2-\alpha (1+C \|\rJ-\Id\|_{L_v^\infty}) \|v\|_{L^2(\partial\cO')}) \leq
 \cQ_{\alpha}[\cO',\rG](v)
  \\
  \leq (1+ C \|\rJ-\Id\|_{L_v^\infty})\|\nabla v \|_{L^2}^2-\alpha (1-C \|\rJ-\Id\|_{L_v^\infty}) \|v\|_{L^2(\partial\cO')})
\end{align*}
That is, for $\|\rJ-\Id\|_{L_v^\infty}$ small enough:
\begin{multline}
\label{E:approxmetgen}
(1- C \|\rJ-\Id\|_{L_v^\infty})\left(\|\nabla v \|_{L^2}^2-\alpha \frac{1+C \|\rJ-\Id\|_{L_v^\infty}}{1- C \|\rJ-\Id\|_{L_v^\infty}} \|v\|_{L^2(\partial\cO')})\right) \leq
 \cQ_{\alpha}[\cO',\rG](v)
\\
  \leq (1+ C \|\rJ-\Id\|_{L_v^\infty})\left(\|\nabla v \|_{L^2}^2-\alpha \frac{1-C \|\rJ-\Id\|_{L_v^\infty}}{1+C \|\rJ-\Id\|_{L_v^\infty}} \|v\|_{L^2(\partial\cO')}\right)
\end{multline}
Similarly, we have a norm approximation: Assuming that $\|\rJ-\Id\|_{L_v^\infty} \leq 1$,
\begin{equation}
\label{E:approxnormgenerique}
\forall v\in L^{2}(\cO'), \quad (1-C \|\rJ-\Id\|_{L_v^\infty}) \|v\|_{L^2} \leq \|v\|_{L^2_{G}} \leq (1+C \|\rJ-\Id\|_{L_v^\infty}) \|v\|_{L^2}.
\end{equation}
By applying the previous inequality on tangent geometries with a constant metric, we will deduce the continuity of the local energy on strata in Section \ref{SS:6.1}.

\subsection{Functions with small support}
The following lemma compares the quadratic form with a metric to the one without metric for functions concentrated near the origin of a tangent cone:
\begin{lemma}
\label{L:Estimatesmetrics}
Let $\Omega \in \gD(M)$, let $x_{0}\in \overline{\Omega}$, and let $\psi_{x_{0}}:\cU_{x_{0}}\mapsto \cV_{x_{0}}$ be a map-neighborhood of $x_{0}$. Let $\rG$ be the associated metric, defined in Lemma \ref{L:CV}. Then there exists universal positive constants $c$ and $C$ such that for all $r \in (0,\frac{c}{\kappa(x_{0})})$ with $\cB(0,r)\subset \cV_{x_{0}}$, for all $v\in H^{1}(\Pi_{x_{0}})$ compactly supported in $\cB(0,r)$, there holds 
\begin{equation}
\label{E:estimateFQ}
(1-Cr\kappa(x_{0}))\cQ_{\alpha^{-}}[\Pi_{x_{0}}](v) \leq \cQ_{\alpha}[\Pi_{x_{0}},\rG](v)  \leq (1+Cr\kappa(x_{0}))\cQ_{\alpha^{+}}[\Pi_{x_{0}}](v)
\end{equation}
where 
\begin{equation}
\label{D:alphapm}
\alpha^{\pm}(r,x_{0})=\alpha\frac{1\mp Cr\kappa(x_{0})}{1\pm Cr\kappa(x_{0})},
\end{equation}
and 
$$|\|v\|_{L^2}-\|v\|_{L^2_{G}}| \leq  Cr\kappa(x_{0})\|v\|_{L^2}.$$
Here $\kappa(x)$
is defined in \eqref{D:curvature}.
\end{lemma}
\begin{proof}
 Let $\rJ$ be the Jacobian of $\psi_{x_{0}}^{-1}$. Since $v$ is supported in a ball $\cB(0,r)$ and $\rJ(0)=\Id$, by direct Taylor inequality we get $\|\rJ-\Id\|_{L^{\infty}(\cB(0,r))} \leq r \|\rJ\|_{W^{1,\infty}(\cO)}=r\kappa(x_{0})$. We use \eqref{E:UBRough}, and we follow the same steps leading to \eqref{E:approxmetgen} and \eqref{E:approxnormgenerique}.
\end{proof}

\begin{remark}
\label{R:ineqeasy}
When the quadratic forms are negative, the above inequality implies
\begin{equation}
\label{E:estimateFQneg}
\cQ_{\alpha^{-}}[\Pi_{x_{0}}](v) \leq \cQ_{\alpha}[\Pi_{x_{0}},\rG](v)  \leq \cQ_{\alpha^{+}}[\Pi_{x_{0}}](v).
\end{equation}
\end{remark}

The following lemma will be useful when studying the essential spectrum of tangent operators:
\begin{lemma}
\label{L:UPEnergeps}
 Let $\Omega\in \gD(M)$ and let $x_{0}\in \overline{\Omega}$ such that $E(\Pi_{x_{0}})$ is finite. Let $\cU_{x_{0}}$ be a map-neighborhood of $x_{0}$. Then 
$$\limsup_{\alpha\to+\infty}\inf_{\substack{u\in H^{1}(\Omega), \ \|u\|=1 \\ \supp(u)\subset \cU_{x_{0}}}} \alpha^{-2}\cQ_{\alpha}[\Omega](u) \leq E(\Pi_{x_{0}}).$$
This property is still true if $\Omega\in \gP_n$.
\end{lemma}
\begin{proof}
Obviously, $E(\Pi_{x_{0}})<0$. Let $\eps>0$ such that $E(\Pi_{x_{0}})+\eps<0$. Note that 
\beq
\label{E:CondepsE}
\frac{E(\Pi_{x_{0}})+\eps}{E(\Pi_{x_{0}})+\frac{\eps}{2}}\in(0,1).
\eeq
The functions of $H^{1}(\Pi_{x_{0}})$ with compact support are dense in $H^{1}(\Pi_{x_{0}})$, therefore there exists $v_{\epsilon} \in H^{1}(\Pi_{x_{0}})$, with compact support, such that $\|v_{\epsilon}\|=1$ and $\cQ[\Pi_{x_{0}}](v_{\epsilon})<E(\Pi_{x_{0}})+\frac{\epsilon}{2}$. 
Let $\cV_{x_{0}}=\psi_{x_{0}}(\cU_{x_{0}})$,
we choose $r>0$ such that
\begin{subnumcases}{}
\label{SC1}
 \cB(0,r)\subset \cV_{x_{0}} \ \ \mbox{and}\ \ r\leq \tfrac{c}{\kappa(x_{0})} \\
 \label{SC2}
(\tfrac{1- Cr\kappa(x_{0})}{1+ Cr\kappa(x_{0})})^2(E(\Pi_{x_{0}})+\tfrac{\eps}{2})<E(\Pi_{x_{0}})+\eps.
\end{subnumcases}
Conditions \eqref{SC1} will allow us to apply Lemma \ref{L:Estimatesmetrics}. Note that \eqref{SC2} is possible because of \eqref{E:CondepsE}. The reason for this last condition will appear later.
The value $\alpha^{+}=\alpha^{+}(x_{0},r)$ is well defined in \eqref{D:alphapm}. 
The (normalized) test function $$v_{\epsilon,\alpha^{+}}(x):=(\alpha^{+})^{\frac{n}{2}}v_{\epsilon}(\alpha^{+} x)$$ satisfies 
\begin{equation}
\label{E:scalingbase}
\cQ_{\alpha^{+}}[\Pi_{x_{0}}](v_{\epsilon,\alpha^{+}})=(\alpha^{+})^2\cQ[\Pi_{x_{0}}](v_{\epsilon}),
\end{equation}
see Lemma \ref{L:scalings}, and its support is 
$$\supp(v_{\epsilon,\alpha^{+}})=(\alpha^{+})^{-1}\supp(v_{\epsilon}).$$
Therefore there exists $\alpha$ large enough such that 
\beq
\label{E:conditalpha}
\supp(v_{\epsilon,\alpha^{+}})\subset \cB(0,r),
\eeq
so we can apply Lemma \ref{L:Estimatesmetrics}.
Therefore, by combining \eqref{E:scalingbase} with estimates \eqref{E:estimateFQ}, we get
\begin{align*}
\cQ_{\alpha}[\Pi_{x_{0}},\rG](v_{\epsilon,\alpha^{+}}) & \leq (1+cr\kappa(x_{0}))\cQ_{\alpha^{+}}[\Pi_{x_{0}}](v_{\eps,\alpha^{+}})
\\
&=(1+cr\kappa(x_{0})) (\alpha^{+})^2\cQ[\Pi_{x_{0}}](v_{\eps})
\\
&\leq (1+cr\kappa(x_{0})) (\alpha^{+})^2(E(\Pi_{x_{0}})+\tfrac{\eps}{2}).
 \end{align*}
Due to \eqref{SC1} and \eqref{E:conditalpha}, we can define
$$u_{\epsilon,\alpha}:=v_{\epsilon,\alpha^{+}}\circ \psi_{x_{0}}^{-1},$$
with $\supp(u_{\eps,\alpha})\subset \cU_{x_{0}}$, and Lemma \ref{L:CV} gives 
$\cQ_{\alpha}[\Omega](u_{\epsilon,\alpha})=\cQ_{\alpha}[\Pi_{x_{0}},G](v_{\epsilon,\alpha})$. Moreover there holds $\|u_{\eps,\alpha}\|^2=\|v\|_{L^2_{G}}^2 \leq 1+Cr\kappa(x_{0})$, therefore keeping in mind that for $\eps$ small enough, $E(\Pi_{x_{0}})+\frac{\eps}{2}<0$, we get
$$\frac{\cQ_{\alpha}[\Omega](u_{\epsilon,\alpha})}{\|u_{\eps,\alpha}\|^2} \leq (\alpha^{+})^2(E(\Pi_{x_{0}})+\tfrac{\eps}{2})=(\tfrac{1- Cr\kappa(x_{0})}{1+ Cr\kappa(x_{0})})^2\alpha^2(E(\Pi_{x_{0}})+\tfrac{\eps}{2}).$$
Setting $u=\frac{u_{\eps,\alpha}}{\|u_{\eps,\alpha}\|}$ and using \eqref{SC2}, we have proved
$$\cQ_{\alpha}[\Omega](u) \leq E(\Pi_{x_{0}})+\eps $$ 
and we get the lemma. Since, locally,  a cone of $ \gP_n$ satisfies the same properties as a corner domain, the above proof works when $\Omega$ is a cone. 
\end{proof}
\begin{remark}
As a direct consequence of the previous lemma, the min-max principle would provide a rough upper bound for $\limsup_{\alpha\to+\infty}\frac{\lambda(\alpha,\Omega)}{\alpha^2}$ by $\sE(\Omega)$. But at this stage, we still don't know whether $\sE(\Omega)$ is finite or not when $\Omega$ is an $n$-dimensional corner domain. 
\end{remark}

\section{Lower bound: multiscale partition of the unity}
\label{S:LB}
In this section, we prove the lower bound of Theorem \ref{T:AsReste} for any domain $\Omega\in \gD(M)$. We insist at this point that this lower bound has interest only when $\sE(\Omega)>-\infty$, which is not proved yet.

It relies on a multiscale partition of the unity of the domain by balls. Near each of these balls, we will perform a change of variable toward the tangent cone at the center of the ball, and we will estimate the remainder. However the curvature of the boundary near each centers of balls may be large as this one is close to a conical point. We will counterbalance this effect by choosing balls of radius smaller in regard with the distance to conical points. 

The following lemma is a consequence of \cite[Section 3.4.4 and Lemma B.1]{BoDauPof14}:
\begin{lemma} 
\label{L:partitionunity}
Let $\Omega\in \gD(M)$ and let $\overline{\nu}_{+}$ be the smallest integer satisfying
$$\forall \dx \in \gC(\Omega), \quad l(\dx) \geq \overline{\nu}_{+} \implies \Pi_{\dx}\ \  \mbox{is polyhedral}.$$
For all sequence of scales $(\delta_{k})_{0 \leq k \leq \overline{\nu}_{+}}$ in $(0,+\infty)$ there exists $h_{0}>0$, an integer $L>0$ and a constant $c(\Omega)>0$ such that for all $h\in (0,h_{0})$, there exists an $h$-dependent finite set of points $\cP \subset\overline{\Omega}$ such that for all $p \in \cP$, there exists $0 \leq k \leq \overline{\nu}_{+}$ with
\begin{itemize}
\item the ball $\cB(p,2h^{\delta_{0}+\ldots+\delta_{k}})$ is contained in a map neighborhood of $p$.
\item the curvature associated with this map neighborhood  (defined by \eqref{D:curvature})  satisfies $$\kappa(p) \leq \frac{c(\Omega)}{h^{\delta_{0}+\ldots+\delta_{k-1}}}.$$
\item $\overline{\Omega}\subset \cup_{p\in \cP}\cB(p,h^{\delta_{0}+\ldots+\delta_{k}})$, and each point of $\overline{\Omega}$ belongs to at most $L$ of these balls.
\end{itemize}
\end{lemma}

We set $h=\alpha^{-1}$ and we now choose a partition of the unity $(\chi_{p})_{p\in \cP}$ associated with the balls provided by the previous lemma: each $\chi_{p}$ is $C^{\infty}$ and is supported in the ball $\cB(p,2\alpha^{-(\delta_{0}+\ldots+\delta_{k})})$, moreover:
\beq
\left\{
\begin{aligned}
& \sum_{p\in \cP} \chi_{p}^2=1  \quad \mbox{on} \ \ \overline{\Omega},
\\
& \sum_{p\in \cP} \|\nabla \chi_{p}\|^2_{\infty} \leq C(\Omega)\alpha^{2\delta} \ \ \mbox{with} \ \ \delta=\delta_{0}+\ldots+\delta_{\overline{\nu}_{+}}.
\end{aligned}
\right.
\ee
We apply the IMS formula together with the uniform estimates of gradients:
\begin{align*}
\cQ_{\alpha}[\Omega](u)&=\sum_{p\in \cP} \cQ_{\alpha}[\Omega](\chi_{p}u)-\sum_{p\in \cP} \|\nabla\chi_{p}u\|^2
\\
&\geq \sum_{p\in \cP} \cQ_{\alpha}[\Omega](\chi_{p}u)-C(\Omega)\alpha^{2\delta} \|u\|^2
\end{align*}
Therefore we are left with the task of estimating $\cQ_{\alpha}[\Omega](\chi_{p}u)$ from below for each $p\in \cP$. Let $\psi_{p}$ be a local map on $\cB(p,2\alpha^{-(\delta_{0}+\ldots+\delta_{k})})$ and $v_{p}:=(\chi_{p}u)\circ \psi_{p}^{-1}$. Let $\rG_{p}$ be the associated metric. Then we deduce from lemmas \ref{L:CV} and \ref{L:Estimatesmetrics} that  (recall that the quadratic forms are negative):
\begin{align*}
\frac{\cQ_{\alpha}[\Omega](\chi_{p}u)}{\|\chi_{p}u\|^2}&=\frac{\cQ_{\alpha}[\Pi_{p},\rG_{p}](v_{p})}{\|v_{p}\|^2_{\rG_{p}}} \geq (1+C\alpha^{-(\delta_{0}+\ldots+\delta_{k})}\kappa(p))\frac{\cQ_{\alpha^{-}}[\Pi_{p}](v_{p})}{\|v_{p}\|^2}
\\
&\geq (1+C\alpha^{-(\delta_{0}+\ldots+\delta_{k})}\kappa(p))(\alpha^{-})^2 E(\Pi_{c}) \geq (1+C'\alpha^{-(\delta_{0}+\ldots+\delta_{k})}\kappa(p))\alpha^2 \sE(\Omega)
\\
&=\alpha^2 \sE(\Omega) +O(\alpha^{2-\delta_{k}}),
\end{align*}
where we have used Lemma \ref{L:partitionunity} to control $\kappa(p)$.

The IMS formula provides
$$\forall u\in H^{1}(\Omega), \quad \cQ_{\alpha}[\Omega](u) \geq \left(\alpha^2 \sE(\Omega)+\sum_{k=0}^{\overline{\nu}_{+}}O(\alpha^{2-\delta_{k}})+O(\alpha^{2\delta}) \right)\|u\|^2.$$
 Recall that $\delta=\sum_{k=0}^{\overline{\nu}_{+}}\delta_{k}$, these remainders are optimized by chosen $\delta_{0}=\ldots=\delta_{\overline{\nu}_{+}}$ and $2-\delta_{0}=2\delta=2(\overline{\nu}_{+}+1)\delta_{0}$, that is $\delta_{0}=\frac{2}{2\overline{\nu}_{+}+3}$. We deduce from the min-max principle that there exists $\alpha_{0}\in \R$ and $C^{-}>0$ such that
 \beq
 \label{E:lowerboundproved}
 \forall \alpha \geq \alpha_{0}, \quad \lambda(\Omega,\alpha) \geq \alpha^2 \sE(\Omega)-C^{-}\alpha^{2-\frac{2}{2\overline{\nu}_{+}+3}},
 \ee
\Bk which is the lower bound of
Theorem \ref{T:AsReste}.

\section{Tangent operator}
\label{S:TO}

In this section we describe the Robin Laplacian on a cone $\Pi$, linking some parts of its spectrum with its section $\omega$. 

\subsection{Semi-boundedness of the operator on tangent cones}
\label{SS:SB}
\begin{lemma}
\label{L:UBloin}
Let $\Pi\in \gP_{n}$ and $\omega$ its section. Let $R\geq0$, and let $u\in H^{1}(\Pi)$ with support in $\complement \cB(0,R)$ \footnote{$R=0$ is included, with $\cB(0,0)=\emptyset$}. 
Then
$$\cQ[\Pi](u) \geq 
\left(\inf_{r>R}\frac{\lambda(\omega,r)}{r^2}\right) \|u\|^2_{L^2(\Pi)}.
$$
\end{lemma}
\begin{proof}
Let $\varphi:(r,\theta)\mapsto r\theta$ be the change of variable from $\R_{+}\times \omega$ into $\Pi$ and denote by $v(r,\theta):=u\circ \varphi^{-1}$ the function associated with the change of variable. 
We have
$$  
\|\nabla u\|^2_{L^2(\Pi)}=\int_{r>R}\left( |\partial_{r}v|^2+\frac{1}{r^2}\|\nabla_{\theta}v(r,\cdot)\|^2_{L^2(\omega)}\right) r^{n-1} \rd r,
$$
therefore
\begin{align*}
\cQ[\Pi](u)
&\geq
\int_{r>R}\frac{1}{r^2}\|\nabla_{\theta} v(r,\cdot)\|_{L^2(\omega)}^2r^{n-1} \rd r  - \int_{r>R}\|v(r,\cdot)\|_{L^2(\partial\omega)}^2  r^{n-2} \rd r
\\
&=\int_{r>R}\frac{1}{r^2}\cQ_{r}[\omega](v(r,\cdot))  r^{n-1} \rd r
\geq \int_{r>R}\frac{1}{r^2}\lambda(\omega,r)\|v(r,\cdot)\|_{L^2(\omega)}^2  r^{n-1} \rd r
\\
&\geq \inf_{r>R}\frac{\lambda(\omega,r)}{r^2}\int_{r>R}\|v(r,\cdot)\|_{L^2(\omega)}^2  r^{n-1} \rd r
\end{align*}
and the lemma follows.
\end{proof}
We now prove the following: 
\begin{lemma}
\label{L:coneavecHyppoH1}
Let $\Pi\in \gP_{n}$ such that its section $\omega$ satisfies $\sE(\omega)>-\infty$. Then $E(\Pi)>-\infty$ and the Robin Laplacian $L[\Pi]$ is well defined as the Friedrichs extension of $Q[\Pi]$ with form domain $D(Q[\Pi])=H^1(\Pi)$. 
\end{lemma}
\begin{proof}
Since $\sE(\omega)$ is supposed to be finite, \eqref{E:lowerboundproved} implies
\begin{equation}\label{LBomega}
\liminf_{r \rightarrow + \infty} \frac{\lambda(\omega,r)}{r^2} \geq \sE(\omega).
\end{equation}
Let $\chi_{1}$ and $\chi_{2}$ be two regular cut-off functions defined on $\R_{+}$ such that $\supp(\chi_{1})\subset [0,2R)$ and $\chi_{1}=1$ on $[0,R]$ and $\chi_{1}^2+\chi_{2}^2=1$. The IMS formula (\cite[Lemma 3.1]{Si82}) provides
\begin{equation}
\label{E:IMSpourri}
\cQ[\Pi](u) =\sum_{i=1,2}\cQ[\Pi](\chi_{i} u)-\sum_{i=1,2}\|\nabla \chi_{i} u\|^2.
\end{equation}
Denote by $D_{0}^{R}$ the set of functions in $H^{1}(\Pi\cap \cB(0,2R))$ supported in $\cB(0,2R)$. Since $\Pi\cap \cB(0,2R)$ is a corner domain, $D_{0}^{R}$ has compact injection into $L^{2}(\partial\Pi\cap \cB(0,2R))$ see \cite[Corollary AA.15]{Dau88}. We deduce the existence of a constant $C_{1}(R)\in \R$ such that
$$\cQ[\Pi](\chi_{1} u) \geq  C_{1}(R) \|\chi_{1}u\|_{L^2(\Pi\cap \cB(0,2R))}^2=C_{1}(R) \|\chi_{1}u\|_{L^2(\Pi)}^2.$$
Let $\eps>0$, from \eqref{LBomega} we deduce the existence of $R>0$ such that 
$$\forall r>R,\quad \frac{\lambda(\omega,r)}{r^2} \geq \sE(\omega)-\eps $$
and therefore Lemma \ref{L:UBloin} gives
$$\cQ(\chi_{2}u) \geq (\sE(\omega)-\eps)\|\chi_{2}u\|^2_{L^2(\Pi)}.$$
There exists $C_{2}>0$ such that for all $R>0$, there holds $\sum_{i}\|\nabla\chi_{i}\|^2 \leq C_{2}R^{-2}$. Therefore we deduce that there exists $C_{3}=C_{3}(R,\eps,\omega)\in \R$ such that
$$\cQ[\Pi](u) \geq C_{3} \| u \|^2_{L^2(\Pi)}.$$
We deduce that the quadratic form is lower semi-bounded and the operator $L[\Pi]$ is well defined as the self-adjoint extension of $\cQ[\Pi]$, and its form domain is $H^{1}[\Pi]$. 
\end{proof}
\Bk

\subsection{Bottom of the essential spectrum for irreducible cones}

Let $\Pi \in \gP_{m}$, with $m\geq n$, and let $\Gamma$ be its reduced cone. In some suitable coordinates, we may write
$$\Pi=\R^{m-n}\times \Gamma$$
with $\Gamma\in \gP_{n}$ an irreducible cone. The associated Robin Laplacian admits the following decomposition: 
\beq
\label{E:decomtensor}
L[\Pi]=-\Delta_{\R^{m-n}}\otimes \Id_{n}+\Id_{m-n}\otimes L[\Gamma]
\ee
In particular 
$$\spec(L[\Pi])=[E(\Gamma),+\infty).$$
Moreover, if $E(\Gamma)$ is a discrete eigenvalue for $L[\Gamma]$, and if $u$ is an associated eigenfunction (with exponential decay), then $\Id\otimes u$ is called an $L^{\infty}$-generalized eigenfunction for $L[\Pi]$ (this is linked to the notion of $L^{\infty}$-spectral pair). Therefore we are led to investigate the bottom of the essential spectrum of $L[\Gamma]$. We prove: 
\begin{lemma}
\label{L:specessprove}
Let $\Gamma\in \Pi_{n}$ be an irreducible cone of section $\omega$ and such that $\sE(\omega)>-\infty$. Then the bottom of the essential spectrum of $L[\Pi]$ is $\sE(\omega)$.
\end{lemma}
\begin{proof}
 From Persson's lemma \cite{Pers60}, the bottom of the essential spectrum of $L[\Gamma]$ is the limit, as $R\to+\infty$, of 
$$\Sigma(R):= \inf_{\substack{\Psi\in  H^{1}(\Gamma), \ \Psi\neq0 \\
  \supp(\Psi)\cap \cB(0,R)=\emptyset}}\frac{\cQ[\Gamma](\Psi)}{\|\Psi\|^2} \, .  $$
 {\sc Lower bound:}
 From Lemma \ref{L:UBloin}, we get directly
$$
\liminf_{R\to+\infty}\Sigma(R) \geq \liminf_{R\to+\infty} \frac{\lambda(\omega,R)}{R^2}
$$ 
and we deduce from \eqref{LBomega} that 
$$\liminf_{R\to+\infty}\Sigma(R) \geq \sE(\omega).$$
\Bk

{\sc Upper bound:}
By scaling, see Lemma \ref{L:scalings}, we immediately have 
$$\Sigma(R)=R^{-2}\inf_{\substack{\Psi\in  H^{1}(\Gamma), \ \Psi\neq0 \\
  \supp(\Psi)\cap \cB(0,1)=\emptyset}}\frac{\cQ_{R}[\Gamma](\Psi)}{\|\Psi\|^2} \, .  $$
Each point $x$ in $\overline{\Gamma}\setminus \overline{\cB(0,1)}$ has a tangent cone $\Pi_{x}$. 
If we denote by $x_{1}:=\frac{x}{|x|}\in \overline{\omega}$, and $C_{x_{1}}$ the tangent cone to $\omega$ at $x_{1}$, then $\Pi_{x}\equiv \R\times C_{x_{1}}$. Therefore by tensor decomposition of the Robin Laplacian (see \eqref{E:decomtensor}), there holds $E(C_{x_{1}})=E(\Pi_{x})$. Thus the finiteness of $\sE(\omega)$ implies the finiteness of $E(\Pi_{x})$ and from Lemma \ref{L:UPEnergeps}, we have 
\beq
\label{E:interRR}
\forall x \in \overline{\Gamma}\setminus \overline{\cB(0,1)} , \quad
\limsup_{R\to+\infty}\Sigma(R) \leq  E(\Pi_{x}).
\ee
Using moreover that 
\beq
\label{E:egaliteconesection}
\inf_{x\in \overline{\Gamma}\setminus \overline{\cB(0,1)} }E(\Pi_{x})=\inf_{x_{1}\in \partial{\omega}}E(C_{x_{1}})=\sE(\omega),
\ee
and, taking the infimum in \eqref{E:interRR} over $x\in \overline{\Gamma}\setminus \overline{\cB(0,1)} $,  we deduce
$$\limsup_{R\to+\infty}\Sigma(R) \leq \sE(\omega),$$
and the lemma follows.
\end{proof}

\section{Infimum of the local energies in corner domains}\label{S:Sel}

\subsection{Finiteness of the infimum of the local energies}\label{SS:6.1}

In this section, we prove the finiteness of $\sE(\Omega)$ for any $\Omega\in \gD(M)$ and for any $n$-dimensional manifold $M$ without boundary, by induction on the dimension $n$. 

In dimension $1$,  bounded domains are intervals and the associated tangent cones are either half-lines or the full line whose associated energies are respectively -1 and 0 (by explicit computations),  therefore the infimum of the local energies is finite.

Let $n\geq2$ be set and let us assume that for any corner domain $\omega$ of an $n-1$ Riemannian manifold without boundary, we have
$$ \sE(\omega) >-\infty.$$
We want to prove that the same holds in dimension $n$.

As a consequence of the recursive hypothesis, $E(\Pi)$ is finite for all $\Pi \in \gP_{n}$, see Lemma \ref{L:coneavecHyppoH1}, and we can study the regularity of the local energy with respect to the geometry of a cone:

\begin{proposition}
Assume the recursive hypothesis in dimension $n-1$. Then the application $\Pi\mapsto E(\Pi)$ is continuous on $\gP_{n}$  for the distance $\dD$ defined in \eqref{D:dd}.
\end{proposition}
\begin{proof}
Let $\Pi\in \gP_{n}$ and $(\Pi_{k})_{k\in \N}$ a sequence of cones with $\dD(\Pi_{k},\Pi)\to0$ as $k\to+\infty$. This means that there exist a sequence $(\rJ_{k})_{k\in \N}$ in $GL_{n}$ with $\rJ_{k}(\Pi_{k})= \Pi$, $\|\rJ_{k}\|\leq1$ and $\|\rJ_{k}-\Id\|\to0$ as $k\to+\infty$. Then as a direct consequence of \eqref{E:approxmetgen} and \eqref{E:approxnormgenerique}, we deduce that 
$$\forall v\in H^{1}(\Pi), \quad \lim_{k\to+\infty}\frac{\cQ[\Pi,\rG_{k}](v)}{\|v\|_{L^2_{\rG_{k}}}^2}=\frac{\cQ[\Pi](v)}{\|v\|^2}.$$
Recall that the form domain of $\cQ[\Pi,\rG_{k}]$ is $H^{1}(\Pi)$, see Section \ref{SS:SB}. \Bk Since $\cQ[\Pi_{k}]$ and $\cQ[\Pi,\rG_{k}]$ are unitarily equivalent (see Lemma \ref{L:CV}), we deduce that $E(\Pi_{k})\to E(\Pi)$ as $k\to+\infty$.
\end{proof}
By definition of the distance on singular chains (see Section \ref{SS:Sc} ), we get:
\begin{corollary}
\label{C:continuity_chains}
Assume the recursive hypothesis in dimension $n-1$. Let $M$ be an $n$ dimensional manifold as above, and let $\Omega\in \gD(M)$ be a corner domain. Then the application $\dx\mapsto E(\Pi_{\dx})$ is continuous on $\gC(\Omega)$ for the distance $\dD$.
In particular, $x\mapsto E(\Pi_{x})$ is continuous on each stratum of $\overline{\Omega}$.
\end{corollary}
Let $M$ be an $n$ dimensional manifold as above, let $\Omega\in \gD(M)$ and let $x_{0}\in \partial\Omega$, in what follows, $\Gamma_{x_{0}}$ is the reduced cone of $\Pi_{x_{0}}$ and $\omega_{x_{0}}\in  \gD(\dS^{d-1})$ its section, with $d\leq n$. We notice that \eqref{E:egaliteconesection} may be written as 
$$\sE(\omega_{x_{0}})=\inf_{x_{1}\in \partial {\omega_{x_0}}}E(\Pi_{(x_{0},x_{1})}).$$
Therefore Lemma \ref{L:coneavecHyppoH1} and \ref{L:specessprove} show that
$$\forall x_{1}\in \overline{\omega_{x_{0}}}, \quad E(\Pi_{x_{0}}) \leq E(\Pi_{(x_{0},x_{1})}),$$
We deduce by immediate recursion:
\begin{corollary}
\label{C:monotonicity}
 Let $\dx_{1}$ and $\dx_{2}$ be two singular chains in $\gC(\Omega)$ satisfying $\dx_{1}\leq \dx_{2}$, we have 
$$E(\Pi_{\dx_{1}}) \leq E (\Pi_{\dx_{2}}).$$
\end{corollary}
We combine this with Corollary \ref{C:continuity_chains} and we can apply \cite[Theorem 3.25]{BoDauPof14} to get  the lower semi-continuity of the local energy function $x\mapsto E(\Pi_{x})$ and from the compactness of $\overline{\Omega}$, we deduce that $\sE(\Omega)$ is finite. This concludes the proof of Theorem \ref{T:inffinite} by induction.

As a consequence, Lemmas \ref{L:coneavecHyppoH1} and \ref{L:specessprove} implies Theorem \ref{T:Efiniessspectrum}.

%
\subsection{Second energy level}
Note that for a cone which is not irreducible, the spectrum consists in essential spectrum, and Theorem \ref{T:Efiniessspectrum} does not apply. However there still exists a threshold in the spectrum: the second energy level of the tangent operator of a cone $\Pi \in \gP_{n}$ is defined as
$$\seE(\Pi):=\inf_{\dx\in \gC_{0}^{*}(\Pi)}E(\Pi_{\dx}),$$
where we recall that $\gC_{0}^{*}(\Pi)$, defined in Section \ref{SS:Sc}, is the set of singular chains of $\Pi$ of the form $\dx=(0,\ldots)$, and with $l(\dx)\geq 2$, where $l(\dx)$ is the length of the chain.

Using Corollary \ref{C:monotonicity} with $\dx_{1}=(0)$, then taking the infimum over the chain $\dx_{2}\geq \dx_{1}$ with $l(\dx_{2})\geq2$, we get $E(\Pi) \leq \sE^{*}(\Pi)$. We also get $\sE^{*}(\Pi)=\inf_{x_{1}\in\partial{\omega}}E(\Pi_{(0,x_{1})})$ and therefore by \eqref{E:egaliteconesection}:
\beq
\label{E:sE=seE}
\sE(\omega)=\seE(\Pi),
\ee
where $\omega$ is the section of the reduced cone of $\Pi$. The quantity $\sE^{*}$ will be discriminant in the analysis brought in Section \ref{S:UB}.

\subsection{Examples}
Inequality $E(\Pi) \leq \sE^{*}(\Pi)$ is strict if and only if the operator on the reduced cone has eigenvalues below the essential spectrum. The presence (or the absence) of discrete spectrum is an interesting question in itself, and we describe here some examples for which this question has been studied. Due to the clear decomposition of the Robin Laplacian on a cone of the form $\R^{m-n}\times \Gamma$, see \eqref{E:decomtensor}, we only treat the case of irreducible cones.

When $\Gamma$ is the half-line, $E(\Gamma)=-1<0=\sE^{*}(\Gamma)$, and an associated eigenfunction is $x\mapsto e^{-x}$. The case of sectors is given by \eqref{E:Esectors}: the inequality is strict if and only if the sector is convex. In that case, an associated eigenfunction is $(x,y)\mapsto e^{-\frac{x}{\sin\theta}}$, where $x$ denotes the variable associated with the axis of symmetry of the sector, and $\theta$ is the opening angle.
  
In \cite{Pank16}, Pankrashkin provides geometrical conditions on three-dimensional cones with regular section. He show that when $\Gamma\in \gP_{3}$ is a cone such that $\R^{3}\setminus \Gamma$ is convex, then $E(\Gamma)=\sE^{*}(\Gamma)$. At the contrary, if $\R^{3}\setminus \Gamma$ is not convex, then $E(\Gamma)$ is a discrete eigenvalue below the essential spectrum. 

Note finally that in \cite{LevPar08}, Levitin and Parnovski uses geometrical parameter to gives a more explicit expression of $E(\Pi)$ when the section of $\Pi$ is a polygonal domain that admits and inscribed circle.
\begin{remark}\label{ContreExample}
 In Theorem \cite[Theorem 3.5]{LevPar08}, it is stated that the bottom of the spectrum of the Robin Laplacian on a cone which contains an hyperplane passing through the origin is -1. The following  example shows that  this statement is incorrect because the bottom of the essential spectrum could be below $-1$:  take a spherical polygon $\omega\subset \dS^{2}$ such that 
\begin{itemize}
\item $\omega$ is included in a hemisphere.
\item $\omega$ has at least a vertex of opening $\theta\in (\pi,2\pi)$.
\end{itemize}
Let $\Pi\subset \R^{3}$ be the cone of section $\omega$, and let $\widetilde{\Pi}$ be its complementary in $\R^{3}$. The cone $\tilde{\Pi}$ contained a half-space, has an edge with opening angle $\tilde{\theta}=2\pi-\theta\in (0,\pi)$. Then, from Theorem \ref{T:Efiniessspectrum} and \eqref{E:Esectors}, we get that the bottom of the essential spectrum of $L[\Pi]$ is below $-\sin^{-2}\frac{\tilde{\theta}}{2}$, and therefore $E(\widetilde{\Pi})<-1$.
\end{remark}
\Bk

\section{Upper bound: construction of quasi-modes}
\label{S:UB}
In order to prove the upper bound of \eqref{T:AsReste}, we construct recursive quasi-modes. The subsections below correspond to the following plan:
\begin{enumerate}
\item Use the analysis of Section \ref{S:Sel} to find a chain $\dx_{\nu}=(x_{0},\ldots,x_{\nu})\in \gC(\Omega)$ such that $L(\Pi_{\dx_{\nu}})$ admits a generalized eigenfunction associated with the value $\sE(\Omega)$. Construct a quasi-mode for $L_{\alpha}[\Pi_{\dx_{\nu}}]$. We do this by using scaling and cut-off functions in a standard way.
\item Use a recursive procedure (together with a multiscale analysis) to construct a quasimode on $\Pi_{x_{0}}$.
\item Use this quasi-mode to construct a final quasi-mode on $\Omega$. Choose the scales to optimize the remainders.
\end{enumerate}
\subsection{A quasi-mode on a tangent structure}
\label{SS:stepA}
The  \Bk next proposition uses the quantity $\seE$ to state that there always exist a tangent structure that admits an $L^{\infty}$-generalized eigenfunction associated with the ground state energy.
\begin{proposition}
\label{P:Dicho}
Let $\Pi\in \gP_{n}$. Then there exists $\dx\in \gC_{0}(\Pi)$ satisfying 
\beq
\label{E:dicho}
E(\Pi_{\dx})=E(\Pi) \quad \mbox{and} \quad E(\Pi_{\dx})<\sE^{*}(\Pi_{\dx}).
\ee
Let $\Gamma_{\dx}\in\gP_{d}$ be the irreducible cone of $\Pi_{\dx}$. Then there exists an $L^{\infty}$-generalized eigenfunction for $L[\Pi_{\dx}]$ associate with $E(\Pi)$. Moreover it has the form $\one\otimes \Psi_{\dx}$, in coordinates associated with the decomposition $\Pi_{\dx}\equiv \R^{n-d}\times \Gamma_{\dx}$, where $\Psi_{\dx}$ has exponential decay. 
\end{proposition}
\begin{proof}
The proof of the existence of $\dx$ is recursive over the dimension $d$ of the reduced cone of $\Pi$. The initialisation is clear, indeed when $d=1$, then $\Pi$ is a half plane, there holds $E(\Pi)=E(\R_{+})=-1$ and $\sE^{*}(\Pi)=E(\R)=0$. Moreover, $\psi_{\dx}(x)=e^{-x}$ provides an  eigenfunction for $L[\R^{+}]$.

We now prove the heredity. First we find a chain $\dx$ satisfying \eqref{E:dicho}
\begin{itemize}
\item
If $E(\Pi)<\sE^{*}(\Pi)$, then $\dx=(0)$ and $\Pi_{\dx}=\Pi$. 
\item
If $E(\Pi) = \sE^{*}(\Pi)$, we use Theorem \ref{T:Efiniessspectrum}: the function $x_{1}\mapsto E(\Pi_{x_{1}})$ is lower semi-continuous on $\overline{\omega}$, where $\omega$ is the section of the reduced cone of $\Pi$. Therefore there exists $x_{1}\in {\partial \omega}$ such that $\sE^{*}(\Pi)=\sE(\omega)=E(\Pi_{x_{1}})=E(\Pi_{(0,x_{1})})$. The dimension of the reduced cone of $\Pi_{(0,x_{1})}$ is lower than the one of $\Pi$, therefore by recursive hypothesis, there exists $\dx' \in \gC_{0}(\Pi_{(0,x_{1})})$ such that $E(\Pi_{\dx'})=E(\Pi_{(0,x_{1})})$ and $E(\Pi_{\dx'})<\sE^{*}(\Pi_{\dx'})$. We write this chain in the form $\dx'=(0,\dx'')$, and the chain $\dx'$ is pulled back into an element of $\gC_{0}(\Pi)$ by setting $\dx=(0,x_{1},\dx'')\in \gC_{0}(\Pi)$. Note that $\Pi_{\dx}=\Pi_{\dx'}$, so that $E(\Pi_{(0,x_{1})})=E(\Pi_{\dx})=\sE^{*}(\Pi)=E(\Pi)$ and $E(\Pi_{\dx})< \sE(\Pi_{\dx})$.
\end{itemize}
From Theorem \ref{T:Efiniessspectrum} and \eqref{E:sE=seE}, $E(\Pi_{\dx})<\sE^{*}(\Pi_{\dx})$ means that $E(\Pi_{\dx})$ is an eigenvalue of $L(\Gamma_{\dx})$ below the essential spectrum, therefore there exists an associated eigenfunction $\Psi_{\dx}$ with exponential decay, and $\R^{n-d}\times \Gamma_{\dx}\ni (y,z)\mapsto \Psi(z)$ is clearly an $L^{\infty}$-generalized eigenfunction for $L[\R^{n-d}\times \Gamma_{\dx}]$.
\end{proof}
\label{SS:stepA}
First, thanks to the lower semi-continuity of local energies, we choose $x_{0}\in \partial\Omega$ such that $E(\Pi_{x_{0}})=\sE(\Omega)$. Then, using Proposition \ref{P:Dicho}, we pick a singular chain $\dx_{\nu}=(x_{0},\ldots,x_{\nu})$ such that $L[\Pi_{\dx_{\nu}}]$ has an generalized eigenfunction associated with $E(\Pi_{x_{0}})$.  We denote by $\dx_{k}=(x_{0},\cdots,x_{k})$ for $0\leq k \leq \nu$, and $\Pi_{k}:=\Pi_{\dx_{k}}$. 

We define 
\beq
\label{D:nubarre}
\overline{\nu}:=\inf\{k \geq0 , \Pi_{k}\ \ \mbox{is polyhedral}\}.
\ee
The index $\overline{\nu}$ provides the shortest chain such that $\Pi_{\overline{\nu}}$ is polyhedral, with $\overline{\nu}=+\infty$ when $\Pi_{\nu}$ is not polyhedral. Moreover when $\overline{\nu}$ is finite, then for all $\overline{\nu} \leq k \leq \nu$, the tangent structure $\Pi_{k}$ is polyhedral, and $\overline{\nu} \leq n-2$, since any chain of length strictly larger than $n-2$ is associated either with a half-space or with the full space.

The tangent structure $\Pi_{\nu}$ writes (in some suitable coordinates) $\R^{p}\times \Gamma_{\nu}$ with $\Gamma_{\nu}$ irreducible. We denote by $\pi_{\Gamma_{\nu}}$ the projection on $\Gamma_{\nu}$ associated with this decomposition. Then, by Proposition \ref{P:Dicho}, there exists an eigenfunction $u$ with exponential decay for $L[\Gamma_{\nu}]$ associated with $E(\Pi_{\nu})$. 

Let $\chi\in \cC^{\infty}(\R^{+})$ be a cut-off function with compact support satisfying 
$$\chi(r)=1 \ \mbox{if} \ r \leq 1  \ \ \mbox{and} \ \ \chi(r)=0 \ \mbox{if} \ r \geq 2.$$
We define the scaled cut-off function
$$\chi_{\alpha}(r)=\chi(\alpha^{\delta}r)$$
where $\delta \in (0,1)$ will be chosen later. The initial quasi-mode writes 
$$u_{\nu}(x)=\chi_{\alpha}(|x|)u(\pi_{\Gamma}(\alpha x)), \quad x\in \Pi_{\dx_{\nu}}.$$
Standard computations show that 
$$\frac{\cQ_{\alpha}[\Pi_{\nu}](u_{\nu})}{\|u_{\nu}\|^2}=\alpha^2\sE(\Omega)+\frac{\|\nabla(\chi_{\alpha})u_{\nu}\|^2}{\|u_{\nu}\|^2},$$
in particular 
\beq
\label{E:EstQMini}
\frac{\cQ_{\alpha}[\Pi_{\nu}](u_{\nu})}{\|u_{\nu}\|^2}=\alpha^2\sE(\Omega)+O(\alpha^{2\delta}).
\ee

\subsection{Getting up along the chains}
The last section provides a quasi-mode $u_{\nu}$ for $L[\Pi_{\nu}]$. The aim of this section is a recursive decreasing procedure in order to get a quasi-mode for $L[\Pi_{0}]$. Therefore, this step is skipped if $\nu=0$. This case happens when $E(\Pi_{x_{0}})<\seE(\Pi_{x_{0}})$, and the quasi-mode is qualified as {\it sitting}, according to the denomination introduced in \cite{BoDauPof14}. Otherwise we suppose that $\nu \geq 1$, and we will construct quasi-modes $u_{k}$ defined on $\Pi_{k}$, for $0 \leq k \leq \nu$. These quasi-modes are qualified as {\it sliding}.

\paragraph{Step A: Recursive construction of quasi-modes}
In what follows, $(d_{k}(\alpha))_{k=1,\ldots,\nu}$ and $(r_{k}(\alpha))_{k=0,\ldots,\nu}$ are positive sequences of shifts and radii (to be determined) going to 0 as $\alpha\to+\infty$.

Let $1 \leq k \leq \nu$ and assume that $u_{k}\in H^{1}(\Pi_{k})$ is constructed, and is supported in a ball $\cB(0,r_{k}(\alpha))$. This is already done for $k=\nu$, see the last section. For $1 \leq k \leq \nu$, we define
$$v_{k}=d_{k}(\alpha)(0,x_{k})\in \Pi_{k-1},$$
where $(0,x_{k})\in \Pi_{k-1}$ are cylindrical coordinates associated with the decomposition $\Pi_{k-1}=\R^{p_{k}}\times \Gamma_{k-1}$. Intuitively, $v_{k}$ is a point of $\Pi_{k-1}$ satisfying $\|v_{k}\|=d_{k}(\alpha)$, and is colinear to $(0,x_{k})$. 

We construct $u_{k-1}$ as follows: 
\begin{itemize}
\item Local map at $v_{k}$. The tangent cone to $\Pi_{k-1}$ at $v_{k}$ is $\Pi_{k}$ itself. Let $\psi_{k}:\cU_{v_{k}}\mapsto \cV_{v_{k}}$ be a local map. The map neighborhoods $\cU_{v_{k}}$ and $\cV_{v_{k}}$ (of $v_{k}\in \Pi_{k-1}$ and $0\in \Pi_{k}$ respectively) can be chosen of diameters smaller than $c_{k}d_{k}(\alpha)$, where $c_{k}$ is the diameter of map-neighborhoods of $x_{k}$. Moreover, when $k \geq \overline{\nu}$, $\Pi_{k}$ is polyhedral so $\psi_{k}$ is a translation. When this is not the case, by elementary scaling, their holds $\kappa(v_{k})\leq\frac{\kappa(x_{k})}{d_{k}(\alpha)}$, see \cite[Section 3]{BoDauPof14} for more details on this process. Since the $(x_{k})_{0 \leq k \leq \nu}$ are fixed, we can choose $\nu$ fixed map neighborhoods associated with these points, and a constant $c(\Omega)>0$ such that
\beq
\label{E:estimekappa}
\kappa(v_{k})\leq 
\left\{
\begin{aligned}
&\frac{c(\Omega)}{d_{k}(\alpha)} \ \ \mbox{if} \ \ k \leq \overline{\nu} 
\\
& c(\Omega) \ \ \mbox{if}\ \ k \geq \overline{\nu}+1
\end{aligned}
\right. .
\ee
 We now add the constraint that 
\beq
\label{E:condrk1}
\frac{r_{k}(\alpha)}{d_{k}(\alpha)}\underset{\alpha\to+\infty}{\to} 0 \ \ \mbox{if} \ \ k \leq \overline{\nu} 
\ee
so that 
$r_{k}\kappa(v_{k})\to0$ for all $1 \leq k \leq \nu$, and we can define for $\alpha$ large enough 
\beq
\label{D:tauk}
\tau_{k}:=\frac{1-Cr_{k}\kappa(v_{k})}{1+Cr_{k}\kappa(v_{k})},
\ee
 where $C$ is the constant appearing in Lemma \ref{L:Estimatesmetrics}.
\item Change of variables. First we rescale $u_{k}$ (the reason for this will appear later) : let
\beq
\label{E:rescaled}
\widetilde{u}_{k}(x)=\tau_{k}(\alpha)^{\frac{n}{2}}u(\tau_{k}(\alpha)x).
\ee
This function satisfies 
\beq
\label{E:scalingQM}
\|\widetilde{u}_{k}\|=\|u_{k} \| \quad \mbox{and} \quad  \cQ_{\alpha_{k}^{+}}[\Pi_{k}](\widetilde{u}_{k})=\tau_{k}(\alpha)^2\cQ_{\alpha}[\Pi_{k}](u_{k})
\ee
where $\alpha^{+}_{k}=\tau_{k}(\alpha)\alpha$. Recall that $\supp(u_{k})\subset \cB(0,r_{k}(\alpha))$ by recursive hypothesis on $u_{k}$. Then due to \eqref{E:condrk1}, we have 
$c_{k}d_{k}(\alpha)>\frac{r_{k}(\alpha)}{\tau_{k}(\alpha)}$ for $\alpha$ large enough, and therefore
$$\supp(\widetilde{u}_{k})\subset\cB(0,\tfrac{r_{k}(\alpha)}{\tau_{k}(\alpha)})\subset \cV_{k}.$$
 As a consequence, we can define on $\cU_{k}\cap\Pi_{k-1}$ the function
\beq
\label{E:ukm1uk}
u_{k-1}=\widetilde{u}_{k}\circ \psi_{k}.
\ee
We can extend this function by 0 outside its support so that $u_{k-1}\in H^{1}(\Pi_{k-1})$. Its support is inside a ball centered at 0 and of size $d_{k}+\diam(\cU_{k})=(1+c_{k})d_{k}$, so we settle 
\beq
\label{E:relrd}
r_{k-1}:=(1+c_{k})d_{k}
\ee
We deduce from this recursive procedure a quasi-mode $u_{0}$ on $\Pi_{0}$, localized in a ball $\cB(0,r_{0}(\alpha))$.
\end{itemize}
\subsection{Quasi-mode on the initial domain $\Omega$ and choice of the scales}
\label{SS:QMscales}
Now we set $v_{0}:=x_{0}$, and we still define $\tau_{0}$ by \eqref{D:tauk}, then $\widetilde{u}_{0}$ by \eqref{E:rescaled} and $u_{-1}$ by \eqref{E:ukm1uk}. Note that $\kappa(v_{0})$ is constant since $v_{0}=x_{0}$ is fixed.
We compare $\cQ_{\alpha}[\Pi_{k-1}](u_{k-1})$ with $\cQ_{\alpha}[\Pi_{k}](u_{k})$ for $0\leq k \leq \nu$. We have from Lemma \ref{L:CV}:
\beq
\label{E:appliCV}
\cQ_{\alpha}[\Pi_{k},\rG_{k}](\widetilde{u}_{k})=\cQ_{\alpha}[\Pi_{k-1}](u_{k-1})
\ee
where $\rG_{k}:=\rJ_{k}^{-1}(\rJ_{k}^{-1})^{\top}$ is the associated metric with $\rJ_{k}:=\rd \psi_{k}^{-1}$.

Since by construction $r_{k}\kappa(v_{k})\to0$, we can apply Lemma \ref{L:Estimatesmetrics}, in particular the inequality \eqref{E:estimateFQneg}:
$$ \cQ_{\alpha}[\Pi_{k},G_{k}](\widetilde{u}_{k})\leq\cQ_{\alpha_{k}^{+}}[\Pi_{k}](\widetilde{u}_{k}).$$
Combining this with identities \eqref{E:scalingQM} and \eqref{E:appliCV} we get for all $0 \leq k \leq \nu$
$$\cQ_{\alpha}[\Pi_{k-1}](u_{k-1}) \leq \tau_{k}(\alpha)^2 \cQ_{\alpha}[\Pi_{k}](u_{k}),$$
and therefore
$$\cQ_{\alpha}[\Omega](u_{-1}) \leq \prod_{k=0}^{\nu}\tau_{k}(\alpha)^2 \cQ_{\alpha}[\Pi_{\nu}](u_{\nu})$$
Recall that $\kappa(v_{0})$ is fixed, we get from \eqref{E:estimekappa}:
$$\cQ_{\alpha}[\Omega](u_{-1}) \leq (1+C(r_{0}+\tfrac{r_{1}}{d_{1}}+\ldots+\tfrac{r_{\overline{\nu}}}{d_{\overline{\nu}}}+r_{\overline{\nu}+1}+\ldots+r_{\nu}))\cQ_{\alpha}[\Pi_{\nu}](u_{\nu}).$$
We now choose $r_{k}(\alpha)=\alpha^{-\sum_{p=0}^{k}\delta_{k}}$ when $k \leq \overline{\nu}$ and $r_{k}=r_{\overline{\nu}}$ when $k>\overline{\nu}$, with $\delta_{k}>0$. The shifts are set by \eqref{E:relrd}, so that $\frac{r_{k}}{d_{k}}=O(\alpha^{-\delta_{k}})$ for all $1 \leq k \leq \overline{\nu}$. 
Moreover, the scale $\delta$ of section \ref{SS:stepA} is related by $\delta=\sum_{k=0}^{\overline{\nu}}\delta_{k}$, and \eqref{E:EstQMini} provides
\begin{align*}
\cQ_{\alpha}[\Omega](u_{-1}) &\leq \left(1+\sum_{k=0}^{\overline{\nu}}O(\alpha^{-\delta_{k}})\right) (\alpha^2\sE(\Omega)+O(\alpha^{2\delta}))
\\
&= \alpha^2 \sE(\Omega)+\sum_{k=0}^{\overline{\nu}}O(\alpha^{2-\delta_{k}})+O(\alpha^{2\delta})
\end{align*}
 The error terms are the same as in Section \ref{SS:QMscales}, therefore we make the same choice of scales $\delta_{k}=\frac{2}{2\overline{\nu}+3}$ for all $0 \leq k\leq \overline{\nu}$. By construction, $u_{-1}$ is normalized, therefore the min-max Theorem implies the upper bound of Theorem \ref{T:AsReste}.
\section{Applications}
\label{S:Applications}
In the applications below, one must keep in mind that the finiteness of $\sE(\Omega)$ is one of our results, and that this quantity can be made more explicit for particular geometries, see \cite{LevPar08}. Moreover, this quantity goes to $-\infty$ as the corners of a domain $\Omega$ gets sharper: this is clear in dimension 2 since the local energy at a corner of opening $\theta$ goes to $-\infty$ as $\theta\to0$, see \eqref{E:Esectors}. In higher dimension, it could be possible to use the approach from \cite{BoDaPoRa15} in order to show that the local energy goes to $-\infty$ for sharp cones (see the definition of a sharp cone therein). \Bk
\subsection{On the optimal constant in relative bounds zero for the trace operator}
 The trace injection from $H^{1}(\Omega)$ into $L^{2}(\partial\Omega)$ being compact, the following relative $0$-bound holds: 
 $\forall \varepsilon >0, \quad \exists C(\varepsilon)>0$, such that:
 \begin{equation}\label{zerobd}
  \| u \|^2_{L^2(\partial \Omega)} \leq \varepsilon \| \nabla u \|_{L^2(\Omega)}^2 +C(\varepsilon)  \|  u \|_{L^2(\Omega)}^2, \qquad \forall u \in H^1(\Omega).
  \end{equation}
This inequality is a particular case of Ehrling's Lemma. It can be written as 
 $$\cQ_{\frac1{\varepsilon}}[\Omega](u) \geq - \frac{C(\varepsilon)}{\varepsilon}  \|  u \|_{L^2(\Omega)}^2, \qquad \forall u \in H^1(\Omega).$$
Thus, by definition of $\lambda(\Omega, \alpha)$, for each $\varepsilon >0$, the best constant $C(\varepsilon)$ in \eqref{zerobd} is:
$$C(\varepsilon) = - \varepsilon \; \lambda(\Omega, \tfrac1{\varepsilon}).$$
From Theorem \ref{T:AsReste}, we obtain that this constant is essentially $\varepsilon^{-1}|\sE (\Omega)|$. More precisely:
\begin{proposition}\label{Appli1}
Let $\Omega \in \gD(M)$ be an admissible corner domain. Then there exist $\varepsilon_{0} >0$ and $\gamma \in (0,\frac23)$,  such that for all $\varepsilon \in (0, \varepsilon_{0})$:
 $$\| u \|^2_{L^2(\partial \Omega)} \leq \varepsilon \| \nabla u \|_{L^2(\Omega)}^2 + \Big( \frac{\vert \sE ( \Omega) \vert}{\varepsilon} + O(\varepsilon^{\gamma-1}) \Big)  \|  u \|_{L^2(\Omega)}^2, \qquad \forall u \in H^1(\Omega).$$
 \end{proposition} 
Let us recall that the finiteness of $\lambda(\Omega, \alpha)$ is closely related to the compactness of the injection of $H^{1}(\Omega)$ into $L^{2}(\partial\Omega)$ and for some cusps, where $\lambda(\Omega, \alpha)=-\infty$, this injection is not compact  (see \cite{NazTask11,Dan13}).

\subsection{Transition temperature of  superconducting models.}

In the study of superconducting models, the physics literature has explored over the years the possibility of increasing the critical fields. Another more interesting and more recent idea is to increase the temperature below which the normal state (i.e. the critical point of the Ginzburg-Landau energy for which the material is nowhere in the superconducting state) is not stable. For zero fields associated to a superconducting body $\Omega$, enhanced surface superconductivity is modeled via a negative penetration depth $b<0$ and following \cite{GiSm07}, this critical temperature, $T_c^b(\Omega)$, is given by:
\beq\label{defTc}
T_c^b(\Omega)=T_{c_0} - T_{c_0} \lambda\Big(\Omega, \frac{\xi(0)}{|b|}\Big)
\ee
where  $\xi(0)>0$ is the so-called coherence length at zero temperature, $T_{c_0}$ is the vacuum zero field critical temperature for $b= \infty$ (corresponding to a superconductor surrounded by vacuum) and $\lambda(\Omega, \alpha)$ is the first eigenvalue of the Robin problem.

Thanks to Theorem \ref{T:AsReste}, for $|b|$ small enough, we have 
$$T_c^{b}(\Omega) \geq  T_{c_0} + T_{c_0}  \frac{\xi(0)^2}{|b|^2} \Big(  \vert \sE(\Omega) \vert   +  O\big( |b|^{\gamma} \big) \Big) $$
for some $\gamma \in (0,\frac23)$. Since $\vert \sE(\Omega) \vert \geq 1$ and goes to $+ \infty$ as the corners of $\partial \Omega$ become sharper, our results is consistent with the general physical principle of increase of $T_c^b(\Omega)$ due to confinement (see for instance  \cite[Section 4]{MontInd00} and see  \cite{YaPe00, BaYaPe02} concerning superconducting properties of nanostructuring materials).

 \Bk

\bibliographystyle{amsalpha}
\bibliography{biblio}

\end{document}